\documentclass{amsart}
\usepackage[utf8]{inputenc}
\usepackage[T1]{fontenc}
\usepackage{lmodern}
\usepackage{amsfonts, amsmath, amssymb,amsthm,enumitem,mathtools,graphicx,array,tikz-cd,microtype}
\usepackage[all]{xy}
\setenumerate[0]{label=(\roman*)}
\usetikzlibrary{fit,backgrounds,3d,matrix,arrows,chains}
\numberwithin{equation}{subsection} 

\theoremstyle{plain}
\newtheorem{lemma}{Lemma}[section]
\newtheorem{theorem}[lemma]{Theorem}
\newtheorem{corollary}[lemma]{Corollary}
\newtheorem{proposition}[lemma]{Proposition}
\theoremstyle{definition}

\newtheorem{definition}[lemma]{Definition}
\newtheorem{remark}[lemma]{Remark}
\newtheorem{example}[lemma]{Example}

\newtheorem{observation}[lemma]{Observation}

\newcommand{\QQ}{\mathbf Q}

\newcommand{\RR}{\mathbf R}
\newcommand{\ZZ}{\mathbf Z}
\newcommand{\FF}{\mathbf F}
\renewcommand{\P}{\mathbf P}
\newcommand{\frob}{\mathrm{Frob}}
\DeclareMathOperator{\codim}{codim}
\DeclareMathOperator{\im}{im}
\DeclareMathOperator{\cone}{cone}
\newcommand{\HH}{\mathrm{H}}
\newcommand{\CH}{{\mathrm{CH}}}
\newcommand{\R}{{\mathrm{R}}}
\newcommand{\GL}{\mathrm{GL}}
\newcommand{\Gr}{\mathrm{Gr}}
\newcommand{\DL}{\mathrm{DL}}
\newcommand{\Fl}{\mathrm{Fl}}
\newcommand{\mfS}{\mathfrak{S}}
\newcommand{\PG}{\mathrm{PG}}
\newcommand{\Kly}{\mathrm{Kly}}
\DeclareFontFamily{U}{mathc}{}
\DeclareFontShape{U}{mathc}{m}{it}%
{<->s*[1.03] mathc10}{}
\DeclareMathAlphabet{\mathcal}{U}{mathc}{m}{it}
\newcommand{\sco}{\mathcal O}
\newcommand{\sci}{\mathcal I}
\newcommand{\scj}{\mathcal J}
\newcommand{\on}{\operatorname}
\renewcommand{\setminus}{\smallsetminus}
\newcommand{\lto}{\xymatrix@1@=15pt{&\ar[l]}}
\newcommand{\lra}{\xymatrix@1@=15pt{\ar[r]&}}
\renewcommand{\mapsto}{\xymatrix@1@=15pt{\ar@{|->}[r]&}}
\newcommand{\mapslto}{\xymatrix@1@=15pt{&\ar@{|->}[l]&}}
\renewcommand{\twoheadrightarrow}{\xymatrix@1@=18pt{\ar@{->>}[r]&}}

\renewcommand{\hookrightarrow}{\xymatrix@1@=15pt{\ar@{^(->}[r]&}}
\newcommand{\hook}{\xymatrix@1@=15pt{\ar@{^(->}[r]&}}

\renewcommand{\emptyset}{\varnothing}
\newcommand{\isoto}{\xymatrix@1@=15pt{\ar[r]^-\sim&}}

\usepackage[pdftex,hyperindex=true, bookmarks=true, bookmarksopen=true, bookmarksnumbered=true,colorlinks=true,citecolor=blue,linkcolor=blue,urlcolor=blue,]{hyperref}
\makeatletter \def\blfootnote{\xdef\@thefnmark{}\@footnotetext} \makeatother

\title[The $q$-Klyachko algebra and geometry]{Deligne--Lusztig varieties, toric orbifolds, and the $q$-Klyachko algebra}
\author{Ruizhen Liu}
\address{Department of Mathematics, University of Toronto, Toronto, Ontario M5S 2E4, Canada}
\email{ruizhen.liu@mail.utoronto.ca}
\date{}
\begin{document}

\begin{abstract}
    We investigate the geometry behind the $q$-Klyachko algebra, introduced by Nadeau--Tewari. When $q$ is a prime power, we show that the $q$-Klyachko algebra is the image of the pullback map on Chow rings $\mathrm{CH}(\mathrm{Fl}_{n+1})\to\mathrm{CH}(\mathrm{DL}_n)$, where $\mathrm{DL}_n\subseteq \mathrm{Fl}_n$ is a compactified Deligne--Lusztig variety inside the complete flag variety $\mathrm{Fl}_{n+1}$. When $q$ is a positive rational number, we establish a K{\"a}hler package for the $q$-Klyachko algebra through inputs from toric geometry.
\end{abstract}
\maketitle
\pagestyle{headings}
\section{Introduction}
Fix $q\in \QQ_{>0}$ and a natural number $n\in \ZZ_{>0}$. In~\cite{NadeauTewari}, Nadeau and Tewari considered:
\begin{definition}
    The \emph{$q$-Klyachko algebra} $\Kly_{n,q}$ is the finite $\QQ$-algebra \[
        \Kly_{n,q} \coloneqq \QQ[u_1,\cdots,u_n]/\langle\,(q+1) u_i^2 = u_i u_{i+1} + q u_i u_{i-1}, \,\text{for all $i=1,\,\ldots,\,n$}\,\rangle,
    \]
    with the convention $u_0=u_{n+1}=0$. It is equipped with a degree map $\deg_{n,q}\colon \Kly^{n}_{n,q}\to \QQ$ by sending $u_1\cdots u_n$ to $(n)_q! = \prod_{i=1}^n \frac{q^i-1}{q-1}$.
\end{definition}
The set of square-free monomials in $u_i$ form a basis for $\Kly_{n,q}$. Thus the degree map is well-defined. The purpose of the present note is to study the geometry behind the $q$-Klyachko algebra.
\subsection{Motivation from probability theory}\label{sec:motivation-probability}
Consider a
compactly supported point measure on the set of integers of the
form $\eta =
\sum_{i\in \ZZ} c_i
\delta_i$, where $c_i\in \ZZ_{\ge 0}$ vanishes for all but finitely many $i$, and $\delta_i$ is the Dirac measure at $i$. 
Fix $q_L,\,q_R\in \QQ_{>0}$ with $q_L+q_R=1$. We consider the following \emph{random displacement} rule. Whenever $\eta(i)\ge 2$ for
some $i$, we set \[
    \eta \mapsto \eta - \delta_i + \delta_{i-1}\text{ with probability $q_L$, or }\eta
    -\delta_i+\delta_{i+1}\text{ with probability $q_R$.}
\]
That is, we move $1$ unit of the mass at $i$ leftward with probability $q_L$ and rightward with
probability $q_R$. See Figure~\ref{fig:probability} for an illustration. A point measure with
$\eta(i)\in\{0,1\}$ for all $i\in \ZZ$ is invariant under this displacement rule.

Let $I\subset \ZZ$ be a finite subset. Denote by $p(I;\eta)$ the probability that the point measure
$\eta$ transforms to the point measure $\sum_{i\in I} \delta_i$ after running the random displacement process above for
sufficiently many steps. Here, at every step, we select one $i$ with $\eta(i)\ge 2$ according to some rule.

In~\cite[\S5]{NadeauTewari}, Nadeau and Tewari showed that the structure constants for the square
monomial basis of $\Kly_{n,q}$, known as \emph{remixed Eulerian numbers}, compute the probability $p(I;\eta)$. More precisely, setting $q =
q_L / q_R$, if $\eta$ is supported on $[n]$ and is such that $\sum_{1\le i \le n} \eta(i) = n$, then
we have
\[
    p([n];\eta) = \frac{1}{(n)_q!}\deg_{n,q} \prod_i u_i^{\eta(i)}.
\]
Moreover, this holds for any rule of selecting $i$ to update.
\begin{figure}
    \centering
    \includegraphics{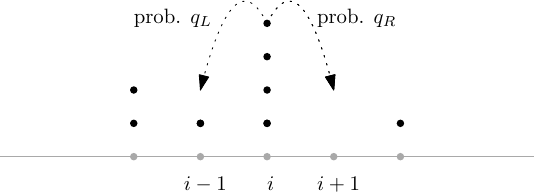}
    \caption{An application of the random displacement rule at $i$.}
    \label{fig:probability}
\end{figure}

\subsection{Statements of the main results}

The first geometric object related to $\Kly_{n,q}$ is a Deligne--Lusztig variety.
Fix a finite field $k = \FF_q$ of order $q$, in addition to an algebraic closure $\bar k$. Denote by $\Fl_{n+1}$ the complete flag variety over $\bar k$. Let $\DL_n$ be the compactified Deligne--Lusztig variety~(\S\ref{sec:DL}) associated to a Coxeter element and the $q$-Frobenius. The space $\DL_n$ is a smooth projective subvariety of $\Fl_{n+1}$; let $\iota\colon \DL_n\to \Fl_{n+1}$ be the evident embedding.

\begin{theorem}\label{th:algebra-isoms}
    For $1\le i\le n$, we set $L_i$ to be the pullback of line bundle on $\Fl_{n+1}$ associated to
    the $i$\/th fundamental weight for the type $A_n$ root system. Then, we have the following isomorphism of graded $\QQ$-algebras, \[
    \epsilon\colon \Kly_{n,q} \isoto \im(\iota^*\colon \CH(\Fl_{n+1})\to\CH(\DL_n)),\quad u_i \mapsto \iota^\ast c_1(L_i).\]
\end{theorem}

Our argument relies on a result of Langer~\cite{Langer}, in which he identified the $\DL_n$ as the
de Concini--Procesi wonderful model for the arrangement $\PG(n,q)$ of all $\FF_q$-hyperplanes, so
$\CH(\DL_n)$ is isomorphic to the matroidal Chow ring $\CH(\PG(n,q))$ in the sense of
Feichtner--Yuzvinsky~\cite{FeichtnerYuzvinsky}. If no confusion arises, we also denote by $\PG(n,q)$ the matroid
associated to the arrangement.

\begin{remark}
    \begin{enumerate}
        \item In \cite[\S7]{NadeauTewari}, Nadeau and Tewari related the $q$-Klyachko algebra $\Kly_{n,q}$ to the Schubert cycle expansion of the variety $\DL_n$. This is an illustration of Theorem~\ref{th:algebra-isoms} at degree $n$.
        \item In~\cite[Remark~7.13]{KatzKutler}, Katz and Kutler exhibited a ring map $\Kly_{n,q}\to\CH(\PG(n,q))$ and asked about the geometrisation of this map. This map is, up to change of variables, given by $\epsilon$. Theorem~\ref{th:algebra-isoms}, therefore, provides a positive answer to the question of Katz--Kutler.
    \end{enumerate}
\end{remark}

In a different direction, we establish a second geometric interpretation for the $q$-Klyachko algebra, valid
even when $q$ is not a prime power. 
\begin{theorem}\label{thm:toric-isom}
    For any positive rational number $q$, there exists a projective simplicial fan $\Sigma_{n,q}$ (\S3) such that the cohomology ring $\HH^\bullet(X_{\Sigma_{n,q}},\QQ)$ of the associated toric variety $X_{\Sigma_{n,q}}$ is isomorphic to $\Kly_{n,q}$ as graded $\QQ$-algebras. Under this isomorphism, the ample cone of $X_{\Sigma_{n,q}}$ contains the cone \[
    \mathcal{K}_{n,q} \coloneqq \RR_{>0}\,u_1 + \cdots + \RR_{>0}\,u_n\subseteq \Kly_{n,q,\RR}^1.
    \]
\end{theorem}
This result is inspired by the work of Abe and Zeng~\cite{abe2023petersonvarietiestoricorbifolds}, who established the case when $q=1$. Granting the theorem, we immediately deduce:
\begin{corollary}\label{thm:kaehlerPackage}
    \begin{enumerate}
        \item (K{\"a}hler package) The $q$-Klyachko algebra $\Kly_{n,q}$ is a Poincar\'e duality algebra. Moreover, it satisfies the Hard Lefschetz theorem and Hodge-Riemann relations with respect to any class $\ell \in \mathcal{K}_{n,q}$.
        \item (Volume polynomial realisation) Keeping the notation and assumptions of Section~\ref{sec:motivation-probability}. We have \[
        (n)_q!\int_{X_{\Sigma_{n,q}}} \left(x_1 u_1 + \cdots + x_n u_n\right)^n = \sum_{\eta} p([n];\eta) \prod_{i=1}^n x_i^{\eta(i)},
    \] where on the left-hand side, the $u_i$ are construed as $\QQ$-divisors on $X_{\Sigma_{n,q}}$, and the right-hand sum runs over all point measures supported on $[n]$ with $\sum_{i=1}^n \eta(i) = n$.
    \end{enumerate}
\end{corollary}

The K{\"a}hler package in Item (i) is a 
consequence of Hodge theory for simplicial polytopes in the sense of Stanley and McMullen~\cite{StanleySimple,McMullenSimple}. Item~(ii), in addition to the Khovanskii--Teissier inequality, implies the probability $p([n];\eta)$ satisfies log-concavity properties of the form \[
p([n];\eta)^2 \ge p([n];\eta-\delta_i+\delta_{i-1}) \cdot p([n];\eta-\delta_i+\delta_{i+1}),\quad \text{for all $i$ with $\eta(i)\ge 2$}.
\]
More broadly speaking, volume polynomials are special cases of Lorentzian polynomials in the sense of Br{\"a}nd{\'e}n and Huh~\cite[Definition~2.1, Theorem~4.6]{BrandenHuh2020}.

\medskip

Finally, we say a few words about the dichotomy between the geometric situations when $q=1$ and when
$q$ is a prime power. The arrangement $\PG(n,q)$ can be seen as the $q$-analogue of the arrangement of coordinate hyperplanes
in $\P^n$. This comes from viewing $\PG(n,q)$ as the vanishing locus of the \emph{Moore
determinant}: \[
    \Delta_{n,q}(x_1,\dots,x_n) = \det
    \begin{pmatrix}
        x_0 & \cdots & x_n\\
        x_0^q & \cdots & x_n^q\\
        \vdots & \ddots & \vdots\\
        x_0^{q^{n}} & \cdots & x_n^{q^n}
    \end{pmatrix}.
\]
Note that replacing the powers $x^{q^i}_j$ by $x_i^{j}$ yields the usual Vandermonde determinant,
which cuts out the coordinate arrangement. Moreover, the wonderful compactification of the coordinate arrangement is the $n$-dimensional permutohedral variety $X_{A_n}$\footnote{Here, we
mean the wonderful model associated to the maximal building set in the sense of~\cite{dcp}.}. From this
perspective, one can view $\DL_n$ as the $q$-analogue of the permutohedral variety $X_{A_n}$, and
Theorem~\ref{th:algebra-isoms} as the $q$-analogue of the well-known isomorphism, usually attributed
to Klyachko~\cite{Klyachko1985}, between $\Kly_{1,n}$ and the image of the restriction map $\CH(\Fl_{n+1})\to \CH(X_{A_n})$.
We summarise the dichotomy in the following table.

\begin{table}[h]
    \centering
    \begin{tabular}{||m{0.45\textwidth}|m{0.45\textwidth}||}
        \hline
        $q=1$ & $q$ prime power \\
        \hline\hline
        Permutohedral toric variety $X_{A_n}$ & Deligne--Lusztig variety $\DL_n$ \\
        \hline
        {\raggedright Coordinate arrangement\par} & {\raggedright The arrangement $\PG(n,q)$ of all $\FF_q$-hyperplanes\par} \\
        \hline
        Klyachko algebra $\Kly_{n,1}$ & $q$-Klyachko algebra $\Kly_{n,q}$ \\
        \hline
        {\raggedright $\Kly_{n,1} \simeq \im(\CH(\Fl_{n+1})\to\CH(X_{A_n}))$ (Klyachko~\cite{Klyachko1985})\par} & {\raggedright $\Kly_{n,q} \simeq \im(\CH(\Fl_{n+1})\to\CH(\DL_n))$ (Theorem~\ref{th:algebra-isoms}) \par}\\
        \hline
        {\raggedright The fan $Y(A_n)=\Sigma_{1,n}$ in~\cite{abe2023petersonvarietiestoricorbifolds,Blume2015,HMSS2024}\par} & The fan $\Sigma_{n,q}$ constructed in~\S\ref{sec:toric-orbifold}. \\
        \hline
        $\Kly_{n,1} \simeq \HH^\bullet(X_{\Sigma_{1,n}},\QQ)$ (Abe--Zeng~\cite{abe2023petersonvarietiestoricorbifolds}) & $\Kly_{n,q} \simeq \HH^\bullet(X_{\Sigma_{n,q}},\QQ)$ (Theorem~\ref{thm:toric-isom}) \\
        \hline
        Peterson variety $Y_n$ & ??? \\
        \hline
    \end{tabular}
\end{table}

For the last row, we note that over the complex numbers, the Peterson variety $Y_n$ has the property that its
cohomology ring is the Klyachko algebra,
$\Kly_{n,1}\simeq \HH^\bullet(Y_n,\QQ)$. In~\cite{abe2023petersonvarietiestoricorbifolds}, Abe and Zeng established a map $Y_n\to X_{\Sigma_{1,n}}$ that realises this isomorphism in cohomology. It is natural to ask if there is a similar subvariety of
$\Fl_{n+1}$ whose cohomology ring equals $\Kly_{n,q}$. The variety $Y_n$ is a regular nilpotent Hessenberg variety with Hessenberg function $h=(2,\,3,\,\dots,\,n,\,n)$. The regular semisimple Hessenberg variety associated to the same Hessenberg function $h$ is exactly the permutohedral variety $X_{A_n}$, whose role is taken by $\DL_n$ in our $q$-analogue situation.

\medskip

Concerning organisation, we start in Section~\ref{sec:DL} explaining the connection among the $q$-Klyachko algebra, the Deligne--Lusztig variety, and the projective geometry matroid, proving Theorem~\ref{th:algebra-isoms}. In \S\ref{sec:toric-orbifold}, we give the construction of the fan $\Sigma_{n,q}$ and prove Theorem~\ref{thm:toric-isom}.

\medskip

\noindent
{\bf Acknowledgements.}
Many thanks to Hunter Spink and Vasu Tewari for comments. Many thanks to Hiraku Abe for explaining his work with Zeng~\cite{abe2023petersonvarietiestoricorbifolds}. The author was partially supported by a University of Toronto Excellence Award.

\section{The Deligne--Lusztig variety of a Coxeter element}\label{sec:DL}
Let us first introduce some notation. First, we define the $q$-integers and $q$-factorials
\[
    (m)_q \coloneqq \frac{1-q^m}{1-q} = 1+\cdots + q^{m-1},\quad (m)_q! \coloneqq \prod_{1\le i\le m} (i)_q.
\]

Put $k = \FF_q$ for $q$ a prime power.
Let $G=\GL_{n+1}$ be the general linear group over $k$ and $B \subseteq G_{\bar k}$ a Borel
subgroup. We define $\Fl_{n+1}:= G_{\bar k} / B$ as the (complete) flag variety over $\bar k$. Let
$\pi_d\colon \Fl_{n+1}\to \Gr(d,n+1)$ be the map that forgets all but the $d$th subspace.
\begin{definition}
    Let $w \in \mfS_{n+1}$. The \emph{Deligne--Lusztig variety}
    $\DL^\circ_n(w)$ associated to $w$ is defined as the fibre product
    \[
        \Gamma_{\frob} \times_{\Fl_{n+1} \times \Fl_{n+1}} [G.(B,wB)]
    \]
    where $\frob\colon \Fl_{n+1} \to \Fl_{n+1}$ is the absolute Frobenius, $G$ acts diagonally on $\Fl_{n+1}\times \Fl_{n+1}$, and $G.(B,wB)$ is the $G$-orbit of $(B,wB)$.
    Now, consider the permutation $c = (1,2)(2,3)\cdots(n,n+1) \in \mfS_{n+1}$. The \emph{Coxeter
    variety} is defined to be $\DL_n^\circ
    \coloneqq \DL^\circ_n(c)$. The \emph{compactified Coxeter variety} $\DL_n$ is closure of
    $\DL_n^\circ$ in $\Fl_{n+1}$.
\end{definition}

The space $\DL_n$ is a smooth projective variety of dimension $n$ with simple normal crossing boundary $\DL_n \setminus \DL_n^\circ$. Analysing the fibre product above, one can describe $\DL_n^\circ$ very explicitly:
\begin{observation}
    Restricting the map $\pi_1\colon \Fl_{n+1} \to \P^n$ yields an isomorphism:
    \begin{align*}
        \pi_1|_{\DL_n^\circ}\colon \DL_n^\circ & \isoto \{P \in \P^n \mid P,\dots,\frob^{n} P\text{ are in general position} \} \\
        & = \P^n \setminus \bigcup\{H: H \text{ is a $k$-rational hyperplane}\},\quad (V_i)\mapsto V_1.
    \end{align*}
    The target is the complement of the hyperplane arrangement $\PG(n,q)$, the arrangement of all $k$-rational hyperplanes in $\P^n$.
\end{observation}

Let $\DL'_n$ be the wonderful compactification of the arrangement complement $\DL_n^\circ$ in the sense of de Concini and Procesi~\cite{dcp}. This is obtained as a series of blow-ups on $\P^n$: first blow up all $k$-points, then the strict transforms of all $k$-lines, and then $k$-planes, and so on. Let $\pi'\colon \DL'_n \to \P^n$ be the composition of these blow-ups.

Langer observed the following; see~\cite[Proposition~7.1]{Langer},
\begin{proposition}[Langer]\label{prop:langer}
    We have the following equality, \[
    \xymatrix{
        \DL \ar[r]^{\pi_1} \ar[d]^{=} & \P^n\ar[d]^{=} \\
        \DL' \ar[r]^{\pi'} & \P^n 
    }\]
    In particular, $\DL_n = \DL_n'$ on the nose.
\end{proposition}

\begin{corollary}
    The rational Chow ring of $\DL_n$ is given by the Chow ring of the matroid $\PG(n,q)$:
    \begin{align*}
        \CH(\DL_n)_\QQ & \simeq \CH(\PG(n,q))_\QQ \\
        & \coloneqq \frac{\QQ[x_F: F\subset \P^n\,\text{a $k$-rational proper subspace}]}{I_{n,q}+J_{n,q}},
    \end{align*}
    where $I_{n,q}$ is generated by linear relations \[
        \sum_{F\subset H} x_F - \sum_{F \subset H'} x_F, \quad\text{for any pair of $k$-rational hyperplanes $H,H' \subset \P^n$,}
    \]
    and $J_{n,q}$ is generated by monomials \[
        x_F x_{F'}\quad \text{for every pair $F,F'$ of subspaces not contained in one another.}
    \]
\end{corollary}
\begin{proof}
    Thanks to Proposition~\ref{prop:langer}, we can compute the Chow ring of $\DL_n$ by viewing it as a wonderful model of hyperplane arrangements. Then the assertion follows immediately from the presentation of the Chow ring of wonderful models established in~\cite[Corollary~2]{FeichtnerYuzvinsky}.
\end{proof}
Here, the class $x_F$ comes from the exceptional divisor of blowing up (the strict transform of) a $k$-linear subspace $F \subset \P^n$. Picking a $k$-rational hyperplane $H\subseteq \P^n$, we define $\alpha\coloneqq \sum_{F\subset H}x_F$---the relations from the ideal $I_{n,q}$ guarantee the linear equivalence class of $\alpha$ is independent of the choice of the hyperplane $H$. Geometrically, this class of $\alpha$ is the first Chern class $c_1(\pi_1^*\sco_{\P^n}(1))$.

Recalling that $N= (n+1)_q-1$, we define divisor classes $\gamma_{k}$ on $X_{A_N}$ as the pullback of the class of the $N$-dimensional hypersimplex $\Delta(N+1,(k)_q)$ in the sense of \cite{BergetSpinkTseng} and \cite{KatzKutler}; geometrically, this comes from pulling back the Pl\"ucker line bundle on the Grassmannian $\Gr((k)_q,N+1)$ via the composite \[
    X_{A_N}\hookrightarrow \Fl_{N+1} \lra \Gr((k)_q,N+1).
\]
We will also denote by $\gamma_k\in \CH^1(\DL_n)$ the restriction of $\gamma_k$ to $\DL_n$.
Also, since $\DL_n$ is a subvariety of $\Fl_{n+1}$, we can define $L_k\in \CH^1 (\DL_n)$ as the restriction of the first Chern class of the $k$\/th Pl\"ucker line bundle on $\Fl_{n+1}$.
In the case when $q=1$ and $\DL_n$ is replaced by the permutohedral variety $X_{A_n}$, the two types of divisors agree; $L_i=\gamma_i$ for all~$1\le i\le n$. The following lemma is a $q$-deformed version of this coincidence.
\begin{lemma}\label{lem:gamma-L-diff}
    For all~$1\le i\le n$, we have $\gamma_i = q^i L_i$ as divisor classes on $\DL_n$.
\end{lemma}
\begin{proof}
    By \cite[Lemma~6.2]{Langer}, we have \[
        L_i = (n+1-i)_q \alpha - \sum_{F\colon \codim F \ge i} (\codim F-i)_q x_F.
    \]
    On the other hand, by \cite[Lemma~2.4]{KatzKutler}, we have \[
        \gamma_i = ((n+1)_q-(i)_q)\alpha -\sum_{F\colon \codim F\ge i}((\codim F)_q - (i)_q) x_F.
    \]
    The claim follows immediately from the identity $\frac{(a)_q-(b)_q}{(a-b)_q} = q^b$.
\end{proof}
Finally, the matroid $\PG(n,q)$ is the special case of a more general class of matroids called \textit{perfect matroid designs}. In~\cite[Lemma~7.10]{KatzKutler}, the authors established quadratic relations among $\gamma_i$ for general perfect matroid designs. Applying their formula to $\PG(n,q)$, we obtain the following.
\begin{lemma}
    The divisor classes $\gamma_i$ satisfy the $q$-Klyachko relation \[
        (q+1)\gamma_i^2 = \gamma_i \gamma_{i+1} + q\gamma_i \gamma_{i-1},\quad\text{for all $1\le i\le n$.}
    \]
\end{lemma}
We are now in a position to establish the main claim.
\begin{proof}[Proof of Theorem~\ref{th:algebra-isoms}]
    First, we check that Pl\"ucker classes $L_i$ also satisfy a $q$-Klyachko relation, we have
    \begin{align*}
        (q+1) L_i^2 & = \frac{q+1}{q^{2i}} \gamma_i^2 = \frac{1}{q^{2i}}(\gamma_i \gamma_{i+1}+ q \gamma_i \gamma_{i-1})\\
        & = \frac{1}{q^{2i}}(q^{2i+1} L_i L_{i+1}+ q^{2i} L_i L_{i-1}) = q L_{i} L_{i+1} + L_i
        L_{i-1},
    \end{align*} by Lemma~\ref{lem:gamma-L-diff}.
    Therefore, the map $\iota^*\colon \CH(\Fl_{n+1}) \to \CH(\DL_n)$ factors through the map $\tau\colon \Kly_{n,q}\to \CH(\DL_n)$ that sends $u_i$ to $L_{n-i}$. Since the Chow ring $\CH(\Fl_{n+1})$ is generated by $L_i$, we get a surjection \[
        \tau\colon \Kly_{n,q} \twoheadrightarrow \im \iota^*.
    \]
    But \cite[Theorem~7.12]{KatzKutler} combined with Lemma~\ref{lem:gamma-L-diff} shows the degree map on $\CH(\DL_n)$ restricts to the degree map on $\Kly_{n,q}$ given by $q$-divided symmetrisation, so $\tau$ must also inject.
\end{proof}

\section{The toric orbifold $X_{\Sigma_{n,q}}$}\label{sec:toric-orbifold}

\subsection{A $q$-deformation of the type $A$ Cartan matrix}
We give a different construction that realises the $q$-Klyachko algebra. The plan is to exhibit the $q$-Klyachko as the Stanley--Reisner ring of a toric orbifold, and then the K{\"a}hler package follows from Hodge theory for simple polytopes.

Assume now $q \in \QQ_{>0}$. Consider the $n$-by-$n$ matrix,
\[
    A(n,q) =
    \begin{pmatrix}
        q+1 & -1 &  & &   \\
        -q & \ddots & \ddots &  &  \\
        & \ddots &  \ddots & -1 \\
        & & -q & q+1
    \end{pmatrix}
\]
where the omitted entries are zero.
When $q=1$, this is the Cartan matrix of the root system of type $A_{n}$. This matrix also appeared
in \cite[Lemma~3.6]{NadeauTewari}. The following list of properties is clear:
\begin{lemma}\label{lem:qCartanSubmatNonsing}
    \begin{enumerate}
        \item The matrix $A(n,q)$ has determinant $(n+1)_q$.
        \item For any subset $J\subseteq [n]$, the submatrix $A(n,q)_{i,j\in J}$ is nonsingular.
        \item The matrix $A(n,q)$ satisfies $(A(n,q)^{-1})_{ij}> 0$ for all $1\le i,j \le n$.
        \item
            For any subset $J\subseteq [n]$, the submatrix $A_J = A(n,q)_{i,j\in J}$ satisfies $(A_J^{-1})_{ij}\ge 0$ for all $i,j\in J$.
    \end{enumerate}
\end{lemma}

Recall that when $q=1$, the matrix $A(n,1)_{i,\,j\in J}$ is a Cartan matrix of the root system of
type $A_n$ for any $J\subseteq [n]$. Then Item~(iv) above is an illustration of the fact that the inverse $[A(n,1)_{i,\,j\in J}]^{-1}$ consists of
nonnegative rational numbers~\cite{LusztigTitsInverseCartan}.

\subsection{The simplicial fan $\Sigma_{n,q}$: construction and properties}

Denote by $e_i$ the standard basis vector in $N=\ZZ^n$, and let $\alpha_1,\,\dots,\,\alpha_n\in \QQ^n$ be the column vectors of $A(n,q)$.

Let $\Sigma_{n,q}$ be the set of cones of the form
\[
    \sigma_{J,K} = \cone\{e_i : i\in J\}-\cone \{\alpha_k : k\in K\}
\] for disjoint subsets $J,K\subsetneq [n]$. Here, our convention is that $\cone \emptyset$ is the origin $\{0\}$, so that $\sigma_{\emptyset,\emptyset} = \{0\}$. As a matter of notation, we write $A =
A(n,q)$ and $A_J = A_{i,j\in J}$.

Using Lemma~\ref{lem:qCartanSubmatNonsing}, one can easily observe:
\begin{lemma}\label{lem:strong-convex-rational}
    For $J,K\in [n]$ so that $J\cap K = \emptyset$, the
    cone $\sigma_{J,K}$ is a strongly convex rational polyhedral cone in $N_\RR$ of dimension $\dim \sigma_{J,K} = \# J + \# K$. Moreover, given $\sigma_{J,K},\sigma_{P,Q}\in \Sigma_{n,q}$, we have $\sigma_{J,K} \cap \sigma_{P,Q} = \sigma_{J\cap P, K\cap Q}$.
\end{lemma}

\begin{corollary}\label{cor:simplicial-fan}
    The set $\Sigma_{n,q}$ is a simplicial fan. In other words, the toric variety $X_{\Sigma_{n,q}}$ associated to the fan $\Sigma_{n,q}$ is $\QQ$-factorial.
\end{corollary}

\begin{proof}
    Firstly, we check that $\Sigma_{n,q}$ is a fan. Indeed, given a cone $\sigma_{J,K}\in
    \Sigma_{n,q}$, every nonempty face is given by the cone $\sigma_{J',K'}\in \Sigma_{n,q}$ for
    $J'\subseteq J$ and $K'\subseteq K$. Secondly, by the second assertion of Lemma~\ref{lem:strong-convex-rational}, the intersection of any two cones $\sigma_{J,K}$ and $\sigma_{P,Q}$ in $\Sigma_{n,q}$ is a face of both, namely, $\sigma_{J,K} \cap \sigma_{P,Q}$.
    Finally, by the first assertion of Lemma~\ref{lem:strong-convex-rational}, the fan $\Sigma_{n,q}$ is simplicial.
\end{proof}

Next, we show the toric variety $X_{\Sigma_{n,q}}$ is projective. The argument goes by first establishing completeness and then detecting amplitude \textit{via} the toric version of Kleiman's criterion.
\begin{lemma}\label{lem:complete-fan}
    The fan $\Sigma_{n,q}$ is complete.
\end{lemma}
\begin{proof}
    The proof here is the same as the one in
    \cite[Proposition~4.3]{abe2023petersonvarietiestoricorbifolds}. The proof in \textit{loc.cit.}
    goes through upon replacing references to
    \cite[Lemma~3.3 and Lemma~3.5]{abe2023petersonvarietiestoricorbifolds} by the first and second
    assertion of Lemma~\ref{lem:strong-convex-rational}.
\end{proof}
Considering the orbit-cone correspondence for toric varieties~\cite[\S3.2]{CLS}, we denote by $D_{-\alpha_i}$ the torus invariant Weil divisor on $X_{\Sigma_{n,q}}$ corresponding to the ray $\sigma_{\emptyset,i}=\cone (-\alpha_i)$ for $i\in [n]$. Similarly, denote by $D_{e_k}$ the torus invariant Weil divisor that corresponds to the ray $\sigma_{k,\emptyset}=\cone e_k$.

\begin{proposition}\label{prop:positive-inter}
    For a torus invariant irreducible curve $C \subseteq X_{\Sigma_{n,q}}$, and an invariant Weil divisor $D_{-\alpha_i}$ that corresponds to the ray generator $-\alpha_i$, the intersection number $(D_{-\alpha_i}\cdot C)$ is nonnegative. Moreover, given an invariant curve $C$, there exists $\ell\in [n]$ such that the corresponding Weil divisor $D_{-\alpha_\ell}$ satisfies $(D_{-\alpha_\ell}\cdot C) > 0$.
\end{proposition}
\begin{proof}
    Under the orbit-cone correspondence, a torus invariant irreducible curve $C$ corresponds to a codimension-one cone of the form $\sigma_{J,K}$ with $J\cup K = [n]\setminus \ell$ for some $\ell$. The cone $\sigma_{J,K}$ is contained in the maximal cones $\sigma_{J\cup \ell,K}$ and $\sigma_{J,K\cup \ell}$. By \cite[Proposition~6.4.4]{CLS}, there exists a unique up to homothety wall relation among the $n+1$ ray generators in $\sigma_{J\cup \ell,K}$ and $\sigma_{J,K\cup \ell}$, given by the linear relation
    \begin{align}\label{eq:wall-relation}
        - (D_{-\alpha_\ell}\cdot C) \alpha_\ell - \sum_{j\in J} (D_{-\alpha_j}\cdot C) \alpha_j + \sum_{k\in K}(D_{e_k}\cdot C) e_k + (D_{e_\ell}\cdot C) e_\ell = 0;
    \end{align} moreover, the intersection numbers $(D_{-\alpha_\ell}\cdot C)$ and $(D_{e_\ell}\cdot C)$ are strictly positive---this proves the second assertion. Put $J' = J\cup \ell$. Pairing Eq.~\eqref{eq:wall-relation} against $e_i$ for $i\in J'$, we obtain a system of linear equations \[
        \sum_{j\in J'} A_{ij} (D_{-\alpha_j}\cdot C) = \delta_{i\ell}\,(D_{e_i}\cdot C)\,\text{for all } i\in J'.
    \] By Lemma~\ref{lem:qCartanSubmatNonsing}, the submatrix $A_{J'}$ is nonsingular, and its inverse $A_{J'}^{-1}$ has nonnegative entries. Solving for the linear equation and noting that $(D_{e_i}\cdot C)>0$, we have $(D_{-\alpha_j}\cdot C)\ge 0$ for all $j\in J'$. This yields the first assertion.
\end{proof}
By Lemma~\ref{lem:complete-fan}, the toric variety $X_{\Sigma_{n,q}}$ is complete. Recall that the toric Kleiman criterion asserts that if $D$ is a Cartier divisor on a complete toric variety $X$, then $D$ is ample if and only if $(D\cdot C) > 0$ for all torus invariant irreducible curves $C$.
Combining this with Proposition~\ref{prop:positive-inter} and the fact that $X_{\Sigma_{n,q}}$ is $\QQ$-factorial (Corollary~\ref{cor:simplicial-fan}), one immediately deduces
\begin{corollary}\label{cor:ample-subcone}
    Given positive rational numbers $a_i\in\QQ_{>0}$ for $1\le i\le n$, the $\QQ$-divisor $D = \sum_{i=1}^n a_i D_{-\alpha_i}$ is $\QQ$-ample. That is, a sufficiently large multiple of $D$ is Cartier and ample.
\end{corollary}
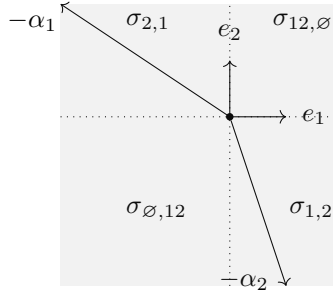
\begin{figure}
    \begin{center}
        \begin{tikzpicture}[scale=0.75]
            \node[circle,fill=black,scale=0.3]  (0) at (0, 0) {};
            \node  (1) at (1, 0) {};
            \node  (2) at (0, 1) {};
            \node  (3) at (-3, 2) {};
            \node  (4) at (1, -3) {};
            \node[below right] at (-2,2) {$\sigma_{2,1}$};
            \node[below left]  (5) at (2, 2) {$\sigma_{12,\emptyset}$};
            \node[above left]  (6) at (2, -2) {$\sigma_{1,2}$};
            \node[above right]  (7) at (-2, -2) {$\sigma_{\emptyset,12}$};
            \node  (8) at (0, 1.5) {$e_2$};
            \node  (9) at (0.25, -2.9) {$-\alpha_2$};
            \node  (10) at (1.5, 0) {$e_1$};
            \node  (11) at (-3.5, 1.7) {$-\alpha_1$};
            \draw[dotted] (-3,0) to (2,0);
            \draw[dotted] (0,-3) to (0,2);
            \draw [->] (0,0) to (1,0);
            \draw [->]  (0,0) to (0,1);
            \draw [->] (0,0) to (-3,2);
            \draw [->] (0) to (1,-3);
            \fill[gray,opacity=0.1] (-3,-3) rectangle (2,2);
        \end{tikzpicture}
    \end{center}
    \caption{Illustration of the fan $\Sigma_{2,2}$, maximal cones labelled by two-part partitions of~$[2]$.}
    \label{fig:KLY22}
\end{figure}
\begin{example}[Case $n=q=2$]
    In Figure~\ref{fig:KLY22}, we visualise the fan $\Sigma_{2,2}$ in $N_\R\simeq\RR^2$. The $4$ maximal cones of $\Sigma_{2,2}$ are the labelled regions. The standard basis vectors $e_i$ generate the cone $\sigma_{i,\emptyset}$, for $i=1,2$. The remaining two vectors are $-\alpha_1 = -3 e_1 + 2 e_2$ and $-\alpha_2 = e_1 - 3e_2$, which gives the cones $\sigma_{\emptyset,1}$ and $\sigma_{\emptyset,2}$ respectively. One observes that $\Sigma_{2,2}$ is the normal fan of a lattice polytope that is combinatorially equivalent to a $2$-dimensional cube.
\end{example}

\subsection{Cohomology of the toric orbifold $X_{\Sigma_{n,q}}$}
This subsection is devoted to proving Theorem~\ref{thm:toric-isom}.
Since the toric variety $X_{\Sigma_{n,q}}$ is $\QQ$-factorial and projective, the rational Betti
cohomology ring $\HH^\bullet(X_{\Sigma_{n,q}},\QQ)$ can be computed by Stanley--Reisner theory~\cite[Chap.~12]{CLS}.
\begin{proof}[Proof of Theorem~\ref{thm:toric-isom}]
    By \cite[Section~12.4]{CLS}, we have the ring isomorphism
    \begin{align*}
        \on{SR}\colon \QQ[X_i,Y_i & : i\in [n]]/(\sci_{n,q}+\scj_{n,q}) \lra \HH^\bullet(X_{\Sigma_{n,q}},\QQ);\\
        X_i & \mapsto [D_{-\alpha_i}],\quad Y_i\mapsto [D_{e_i}],\quad\text{for all $i\in [n]$,}
    \end{align*}
    where the ideal $\sci_{n,q}$ is generated by quadrics $X_iY_i$ for all $i\in [n]$, and the ideal $\scj_{n,q}$ is generated by linear forms \(-\sum_{1\le j\le n} A(n,q)_{ij} X_j + Y_i,\) for all $i\in [n]$.
    By our definition of $A(n,q)$, the ring map $\on{SR}$ factors through an isomorphism \[
        \on{SR}'\colon \QQ[X_1,\dots,X_n]/\sci'_{n,q} \lra \HH^\bullet (X_{\Sigma_{n,q}},\QQ),
        \quad X_i \mapsto [D_{-\alpha_i}],
    \] where the ideal $\sci_{n,q}'\subseteq \QQ[X_1,\dots,X_n]$ is generated by quadrics \[
        X_i\,((q+1)X_i - X_{i+1} - q X_{i-1}),\quad\text{for $1\le i\le n$,}
    \] where we set $X_{0} = X_{n+1} = 0$. But the $q$-Klyachko algebra
    $\Kly_{n,q}$ is isomorphic to the source $\QQ[X_1,\dots,X_n]/\sci_{n,q}'$ of $\on{SR}'$ by
    sending $u_i$ to $X_i$. The second assertion of the theorem follows from Corollary~\ref{cor:ample-subcone}.
\end{proof}

\subsection{A K{\"a}hler package for the $q$-Klyachko algebra}

In this subsection, we use the Hodge theory for simplicial polytopes to prove Corollary~\ref{thm:kaehlerPackage}.
\begin{definition}\label{def:kaehlerPackage}
    Let $A^\bullet = \bigoplus_{i=0}^d A^i$ be a graded Artinian  $\QQ$-algebra of finite type, in addition to an isomorphism $\deg\colon A^d \to \QQ$. We say that $(A^\bullet,\deg)$ \textit{satisfies the K{\"a}hler package} with respect to an element $\ell\in A^1$ if the following conditions are satisfied,
    \begin{enumerate}
        \item (Poincar{\'e} duality) For $0\le k\le d/2$, the bilinear form \[
                A^k \times A^{d-k} \lra \QQ,\quad (\eta_1,\eta_2)\mapsto \deg \eta_1\eta_2
            \] is a perfect pairing.
        \item (Hard Lefschetz theorem) For every integer $0\le k\le d/2$, the multiplication map
            \[
                \times \ell^{d-2k}\colon A^k \lra A^{d-k},\quad \eta \mapsto \ell^{d-2k}\eta
            \] is an isomorphism.
        \item (Hodge index theorem) For every integer $0\le k\le d/2$, the bilinear form \[
                A^k \times A^k \lra \QQ,\quad (\eta_1,\eta_2) \mapsto (-1)^k \deg(\ell^{d-2k}\eta_1\eta_2)
            \] is positive definite upon restricting to the kernel of multiplication by $\ell^{d-2k+1}$.
    \end{enumerate}
\end{definition}

\begin{proof}[Proof of Corollary~\ref{thm:kaehlerPackage}]
    For a simplicial lattice polytope $P$ with normal fan $\Sigma_P$, Stanley and McMullen~\cite{StanleySimple,McMullenSimple} showed that on the toric variety $X_{\Sigma_P}$ associated to the fan $\Sigma_P$, the rational Betti cohomology ring $\HH^\bullet (X_{\Sigma_P},\QQ)$ satisfies Poincar{\'e} duality and the hard Lefschetz theorem with respect to any ample divisor class $\ell \in \HH^2 (X_{\Sigma_P},\QQ)$. Moreover, in \textit{loc.cit.}, McMullen also established the Hodge index theorem for any ample class $\ell$.
    Since the fan $\Sigma_{n,q}$ is projective and simplicial, Item~(i) follows. Concerning Item~(ii), by~\cite[Lemma~12.5.2]{CLS}, we have $
    \int_{X_{\Sigma_{n,q}}} D_{-\alpha_1}\cdots D_{-\alpha_n}  = {|\det A(n,q)|}^{-1} = {(n)_q!}^{-1}$.
    Therefore we have \[\int_{X_{n,q}} \prod_i D_{-\alpha_i}^{\eta(i)} = \frac{1}{[(n)_q!]^{2}}\deg_{n,q} \prod_i u_i^{\eta(i)} = \frac{p([n];\eta)}{(n)_q!}\] for any $\eta\colon [n]\to \ZZ_{\ge 0}$ with $\sum_i \eta(i) = n$. 
\end{proof}
\bibliographystyle{alpha}
\bibliography{refs}
\end{document}